\documentclass[11pt,a4paper]{amsart}
\usepackage{graphicx,multirow,array,amsmath,amssymb}

\usepackage{xcolor}
\usepackage{longtable}
\usepackage{supertabular} 

\newtheorem{theorem}{Theorem}

\newtheorem{lemma}[theorem]{Lemma}
\newtheorem{corollary}[theorem]{Corollary}

\newtheorem{definition}{Definition}

\newtheorem{conjecture}{Conjecture}

\begin{document}
	
\title{A diagrammatic approach to the three-page index}
	
\author[H. Yoo]{Hyungkee Yoo}
\address{Department of Mathematics Education, Sunchon National University, Sunchon 57922, Korea}
\email{hyungkee@scnu.sc.kr}
	
\begin{abstract}
The three-page index $\alpha_3(L)$ is an invariant that measures the complexity
of representing a link $L$ in a three-page book.
It is known that $\alpha_3(L)$ admits a linear upper bound in terms of the crossing number,
with equality realized by the Hopf link.

In this paper, we investigate the equality case of this bound from a diagrammatic viewpoint.
Starting from a reduced link diagram, we construct three-page presentations
via binding circles arising as boundaries of suitable contractible subcomplexes
of the induced cell decomposition of the $2$-sphere.
This approach allows a refined control of the number of arcs
in the resulting three-page presentation.

As a consequence, we prove that for any non-split, nontrivial link $L$
other than the Hopf link,
\[
\alpha_3(L)\le 3c(L)-1,
\]
and hence characterize completely the links for which
$\alpha_3(L)=3c(L)$.
\end{abstract}

\maketitle
\section{Introduction} \label{sec:intro}

Structured embeddings of links in $\mathbb{R}^3$ provide a combinatorial way to study topological complexity.
One influential approach is the arc presentation introduced by Cromwell~\cite{Cr}, in which a link $L$ is embedded into finitely many pages of an open book decomposition of $\mathbb{R}^3$.
Here the open book decomposition consists of a \textit{binding axis} and \textit{pages}, and each page contains exactly one arc.
The minimal number of pages required for such a representation, known as the \textit{arc index} $\alpha(L)$, has become a fundamental invariant closely related to geometric and combinatorial aspects of links~\cite{BP,Cr2,CrNut,Mat,MB,Ng}.

A different perspective was later proposed by Dynnikov~\cite{Dyn}, who introduced \textit{three-page presentations} of links by restricting the ambient space to three pages of an open book decomposition and allowing each page to contain several mutually disjoint arcs.
Dynnikov proved that every link admits a three-page presentation, and subsequent work has investigated such presentations for various classes of links and spatial graphs~\cite{Dyn2, Dyn3, JLY, Kur, Kur2, KV, Y}.
As with the arc index, the \textit{three-page index} $\alpha_3(L)$, introduced in~\cite{JLY}, is defined as the minimal number of arcs required to represent a link in a three-page presentation.

\begin{figure}[h!]
	\includegraphics{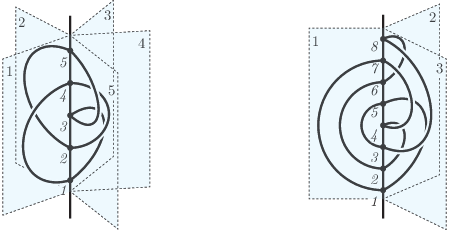}
	\caption{An arc presentation and a three-page presentation of the trefoil knot}
	\label{fig:arc}
\end{figure}

A \textit{link projection} is a regular embedding into $\mathbb{R}^2$, and a \textit{link diagram} is obtained by assigning over–under information at each crossing.
The \textit{crossing number} $c(L)$ of a link $L$ is defined as the minimal number of crossings among all link diagrams representing $L$.
It is one of the most fundamental invariants in link theory, as it provides a basic measure of diagrammatic complexity and serves as a reference scale for many other link invariants.
In particular, numerous geometric and combinatorial invariants are studied in relation to the crossing number, often through upper or lower bounds expressed in terms of $c(L)$.

According to Bae and Park~\cite{BP}, the arc index admits a linear upper bound in terms of the crossing number.
In contrast, for three-page presentations, linear upper bounds have been established only for certain classes of links,
including torus links, rational links, and pretzel links~\cite{JLY, Y}.
It was shown in a previous work~\cite{Y2} of the author that the three-page index
admits a linear upper bound in terms of the crossing number, namely
$$
\alpha_3(L)\le 3c(L).
$$
The sharpness of the inequality is demonstrated by the Hopf link $2_1^2$,
which satisfies $\alpha_3(2_1^2)=6$~\cite{JLY}.
This value is forced by the general lower bound
$\alpha_3(L)\ge 3\,br(L)$,
where $br(L)$ denotes the bridge number.
Since the Hopf link is a two-bridge link, equality follows.

The purpose of the present paper is to sharpen this bound by investigating
the equality case.
We prove that equality is attained only for split unions of Hopf links,
and that every other nontrivial link satisfies a strict inequality.

\begin{theorem}\label{thm:main}
If $L$ is a non-split, nontrivial link other than the Hopf link, then
$$
\alpha_3(L)\le 3c(L)-1.
$$
\end{theorem}

Theorem~\ref{thm:main} shows that, apart from the Hopf link,
the linear bound $\alpha_3(L)\le 3c(L)$ can always be improved by at least one.
Together with the fact that the Hopf link realizes equality
and that both the three-page index and the crossing number
are additive under split union,
this leads to a complete characterization of the equality case.

\begin{corollary}
Let $L$ be a link with no trivial split components.
Then $\alpha_3(L) = 3c(L)$ if and only if $L$ is a split union of Hopf links.
\end{corollary}

The upper bound in Theorem~\ref{thm:main} is expected to be sharp.
In particular, it is conjectured that the three-page index of the trefoil knot
and the figure-eight knot are $8$ and $11$, respectively.
At present, however, we do not have a complete proof of these values.

The proof of Theorem~\ref{thm:main} is constructive.
Starting from a reduced link diagram,
we build a three-page presentation by introducing a suitable binding circle.
Viewing the diagram as a cellular decomposition of the $2$-sphere,
we refine this construction using a contractible subcomplex of the associated
cell complex.
This approach allows us to control the number of arcs appearing
in the induced three-page presentation and leads to the stated bound.

The paper is organized as follows.
In Section~\ref{sec:circular}, we describe a uniform method for constructing
three-page presentations from link diagrams.
In Section~\ref{sec:cell_cpx}, we reinterpret this construction using
cell complexes and prove Theorem~\ref{thm:main}.
Finally, in Section~\ref{sec:graph}, we present a graph-theoretic interpretation
of the construction and discuss possible directions for further improvements.

\section{Circular three-page presentation} \label{sec:circular}

We consider the one-point compactification of $\mathbb{R}^3$ equipped with an open-book structure.
Under this compactification, the binding axis becomes an unknotted circle, and each half-plane becomes a disk.
This circle is called the \textit{binding circle}.
In this setting, the notion of a \textit{circular three-page presentation} was introduced in~\cite{Y}.
The \textit{binding index}, which was originally represented by the vertical height along the binding axis, is now assigned sequentially along the binding circle.

A circular three-page presentation admits a planar description in terms of a link diagram.
To this end, we fix a simple closed curve in the plane and regard it as the binding circle.
By placing the first and second pages inside the binding circle and the third page outside, all crossings can be arranged in the interior region.
In this configuration, strands on the first page appear as under-crossings, while strands on the second page correspond to over-crossings.
With this convention, the page order is fixed, and the binding index is assigned according to the clockwise order along the binding circle, as shown in Figure~\ref{fig:three}.
Conversely, given a simple closed curve presented in this form, one can recover a corresponding three-page presentation.

\begin{figure}[h!]
	\includegraphics{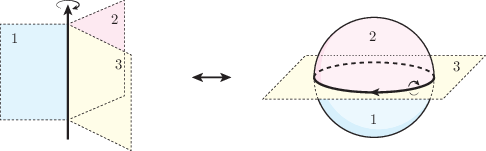}
	\caption{A three-page presentation and its circular three-page presentation with an ordering of the pages}
	\label{fig:three}
\end{figure}

This planar description can be rigorously formalized using the following definition introduced in~\cite{Y2}.

\begin{definition} \label{def:binding-circle}
Let $D$ be a link diagram.
A simple closed curve $\gamma$ is called a \textbf{binding circle for a three-page presentation} of $D$ if it satisfies the following conditions:
\begin{enumerate}
    \item The curve $\gamma$ intersects $D$ in finitely many points.
    \item All crossings of $D$ lie inside $\gamma$.
    \item Each arc of $D$ cut by $\gamma$ is of one of the following three types:
    \begin{enumerate}
        \item it lies outside $\gamma$;
        \item it lies inside $\gamma$ and participates only in over-crossings;
        \item it lies inside $\gamma$ and participates only in under-crossings.
    \end{enumerate}
    \item No two arcs of the same type are adjacent along $\gamma$.
\end{enumerate}
\end{definition}

We now show that a binding circle satisfying the above definition can be constructed for any link diagram.
An intersection point of a binding circle $\gamma$ with the diagram $D$ is called a \textit{binding point}.
By construction, the number of binding points is equal to the number of arcs appearing in the associated three-page presentation.

\begin{lemma} \label{lem:binding-circle-existence}
Let $D$ be a link diagram with $n$ crossings.
Then there exists a binding circle $\gamma$ that induces a three-page presentation with at most $3n+1$ arcs.
\end{lemma}

\begin{proof}
Let $D$ be a link diagram with $n$ crossings.
By discarding the crossing information, we regard $D$ as a planar $4$-valent graph $\Gamma$ with $n$ vertices.

Assume first that the underlying graph $\Gamma$ is connected, and fix a spanning tree $T$ of $\Gamma$.
We take the boundary of a sufficiently small regular neighborhood of $T$ in $\mathbb{R}^2$.
This boundary is a simple closed curve, which we denote by $\gamma$.
After a small perturbation, we may assume that $\gamma$ intersects each edge of the spanning tree exactly once.

Since $\Gamma$ is $4$-valent, each vertex contributes four incident edges.
The curve $\gamma$ intersects, in a neighborhood of each vertex,
all incident edges that are not contained in $T$ exactly once,
and intersects each edge of $T$ once.
As a spanning tree on $n$ vertices has exactly $n-1$ edges,
the total number of intersection points between $\gamma$ and the diagram $D$ is
$$4n-(n-1)=3n+1.$$

We now verify that $\gamma$ satisfies the defining conditions of a binding circle for a three-page presentation.
The first three conditions follow directly from the construction.
It remains to verify the adjacency condition.

If the diagram $D$ is alternating, then adjacent arcs along $\gamma$ alternate between over-crossings and under-crossings.
Hence, $\gamma$ satisfies the fourth condition of Definition~\ref{def:binding-circle}.
If $D$ is non-alternating, then there exists at least one edge of $\Gamma$ whose endpoints correspond to crossings of the same type.
In this case, we slightly modify $\gamma$ near that edge so that it avoids intersecting it.
The resulting curve satisfies all the defining conditions of a binding circle,
and the number of intersection points is reduced.
Therefore, in all cases, we obtain a binding circle inducing a three-page presentation with at most $3n+1$ arcs as drawn in Figure~\ref{fig:spanning}.

Now assume that the underlying planar graph $\Gamma$ has several connected components.
Then the diagram $D$ is split.
We construct such a binding circle for each component and then combine them by successive connected sums,
which yields a single binding circle satisfying the required properties.
The connected sum operation preserves the adjacency condition in Definition~\ref{def:binding-circle}.
\end{proof}

\begin{figure}[h!]
	\includegraphics{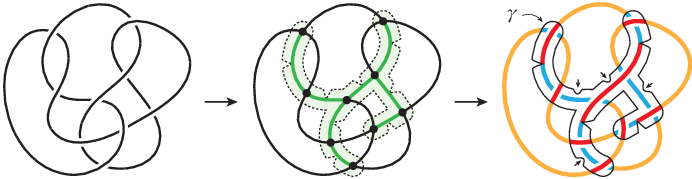}
	\caption{A binding circle associated with a spanning tree of the underlying graph of a link diagram}
	\label{fig:spanning}
\end{figure}

Lemma~\ref{lem:binding-circle-existence} already shows that the three-page index admits a linear upper bound in terms of the crossing number.
However, this bound can be improved by refining the construction.
In the subsequent sections, we describe a method to reduce the coefficient of the linear bound by analyzing the dual graph of the underlying graph associated with a given diagram.

\section{Cell complexes from link diagrams} \label{sec:cell_cpx}

Let $D$ be a non-split reduced link diagram on $S^2$ with $n$ crossings.
The diagram induces a natural cell complex structure on $S^2$.
Let $X(D)$ denote the induced $2$-dimensional cell complex, defined as follows:
\begin{itemize}
\item the $0$-cells correspond to the crossings of $D$,
\item the $1$-cells correspond to the arcs between crossings,
\item the $2$-cells correspond to the complementary regions of $S^2 \setminus D$.
\end{itemize}
Since the underlying graph of $D$ is the $1$-skeleton $X^{(1)}(D)$ of the cell complex $X(D)$, any spanning tree of the underlying graph is a connected and contractible subcomplex of $X^{(1)}(D)$.
The key idea of Lemma~\ref{lem:binding-circle-existence} is that a binding circle can be obtained
as the boundary of a regular neighborhood of a spanning tree of the underlying graph.
We now seek to extend this idea to the level of the cell complex $X(D)$.

For the boundary of a regular neighborhood to be a simple closed curve, the chosen subcomplex must be connected and contractible.
We further require that the chosen subcomplex contain no pair of $2$-cells sharing a common $1$-cell.
After a sufficiently small perturbation, this condition ensures that the boundary of a regular neighborhood of the subcomplex intersects every $1$-cell.
As a result, the resulting simple closed curve functions as a binding circle that distinguishes the three types of arcs in Definition~\ref{def:binding-circle}.
Motivated by these considerations, we introduce the notion of an \emph{extended spanning tree} of $D$, defined as a subcomplex of $X(D)$ satisfying the above conditions.

\begin{definition}
A subcomplex $Y \subset X(D)$ is called an \textbf{extended spanning tree} of a link diagram $D$ if it satisfies the following conditions:
\begin{enumerate}
\item $Y$ is connected and contractible.
\item $Y$ contains all $0$-cells of $X(D)$.
\item No two $2$-cells contained in $Y$ share a common $1$-cell.
\end{enumerate}
\end{definition}

We now aim to improve the result of Lemma~\ref{lem:binding-circle-existence}.

\begin{lemma} \label{lem:extended}
Let $D$ be a non-split reduced link diagram with $n$ crossings, and let $Y$ be an extended spanning tree of $D$.
Then there exists a binding circle $\gamma$ that induces a three-page presentation of $D$ with at most $3n+1-m$ arcs,
where $m$ denotes the number of $2$-cells contained in $Y$.
\end{lemma}

\begin{proof}
Let $D$ be a non-split reduced link diagram with $n$ crossings, and let $Y \subset X(D)$ be an extended spanning tree.
Write $m$ for the number of $2$-cells contained in $Y$.
We construct a binding circle $\gamma$ as the boundary of a sufficiently small regular neighborhood of $Y$ in $S^2$.

Since $Y$ is connected and contractible, a regular neighborhood $\mathcal{N}(Y)$ is a topological disk.
Hence its boundary $\partial \mathcal{N}(Y)$ is a simple closed curve.
After a sufficiently small perturbation, we may assume that $\gamma:=\partial \mathcal{N}(Y)$ is transverse to $X^{(1)}(D)$ and meets each $1$-cell in finitely many points.

We now estimate the number of intersection points between $\gamma$ and the diagram $D$.
Each $0$-cell (crossing) of $X(D)$ has valence $4$ in the $1$-skeleton, so there are $4n$ local half-edges in total.
If $\gamma$ is taken as in Lemma~\ref{lem:binding-circle-existence} from a spanning tree, then one obtains $4n-(n-1)=3n+1$ intersections.
We explain how the presence of $2$-cells in $Y$ decreases this count.

Choose a spanning tree $T$ of the $1$-skeleton $Y^{(1)}$.
Because $Y$ is connected and contractible and contains all $0$-cells, we may view $T$ as a spanning tree of the underlying graph of $D$.
Let $\gamma_0$ be the binding circle obtained as the boundary of a sufficiently small regular neighborhood of $T$.
As in Lemma~\ref{lem:binding-circle-existence}, after a perturbation we may assume that $\gamma_0$ intersects every edge of $T$ exactly once and intersects each edge not in $T$ at most once.
In particular, $\gamma_0$ yields a three-page presentation with at most $3n+1$ arcs.

We now modify $\gamma_0$ using the $2$-cells contained in $Y$.
Let $f$ be a $2$-cell of $Y$.
By the definition of an extended spanning tree, no two $2$-cells of $Y$ share a common $1$-cell.
Hence we may perform the following local modification independently for different $2$-cells.

Consider a small neighborhood of $f$.
The boundary $\partial f$ is a cycle in the $1$-skeleton $X^{(1)}(D)$.
Since $f \subset Y$, the regular neighborhood $\mathcal{N}(Y)$ contains a neighborhood of $f$,
so the curve $\gamma=\partial \mathcal{N}(Y)$ runs outside $\partial f$ and, after a small perturbation, intersects exactly once the unique $1$-cell of $\partial f$ that was previously crossed by $\gamma_0$ in order to pass between the two sides of $\partial f$.
Pushing $\gamma_0$ across $f$ replaces this intersection by a segment of $\gamma$ lying in the interior of $f$ and therefore removes exactly one intersection point with the diagram.
Because distinct $2$-cells in $Y$ do not share a $1$-cell, these $m$ local modifications remove $m$ distinct intersection points.

After performing this modification for every $2$-cell in $Y$, we obtain a simple closed curve $\gamma=\partial \mathcal{N}(Y)$ such that
$$
|\gamma \cap D| \le (3n+1) - m.
$$
By construction, $\gamma$ is the boundary of a regular neighborhood of a connected, contractible subcomplex, so it satisfies the first three conditions in Definition~\ref{def:binding-circle}.
Moreover, since $Y$ contains no pair of $2$-cells sharing a common $1$-cell, a sufficiently small perturbation ensures that $\gamma$ meets every $1$-cell of $X(D)$, and the adjacency condition in Definition~\ref{def:binding-circle} follows exactly as in the proof of Lemma~\ref{lem:binding-circle-existence} (alternating case), with the same local avoidance modification in the non-alternating case.

Therefore, $\gamma$ is a binding circle inducing a three-page presentation of $D$ with at most $3n+1-m$ arcs.
\end{proof}

\begin{figure}[h!]
	\includegraphics{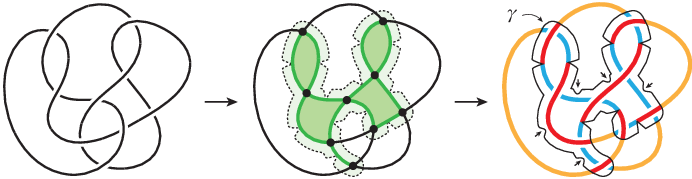}
	\caption{A binding circle associated with an extended 
    spanning tree of the underlying graph of a link diagram}
	\label{fig:extended}
\end{figure}

Lemma~\ref{lem:extended} shows that the number of arcs in a three-page presentation induced by a binding circle
can be reduced by incorporating $2$-cells into the spanning tree.
More precisely, each $2$-cell contained in an extended spanning tree eliminates exactly one intersection
between the binding circle and the diagram.
Thus, improving the linear bound on the three-page index reduces to maximizing
the number of $2$-cells that can be included in an extended spanning tree.
Now we prove the main theorem.

\begin{proof}[Proof of Theorem~\ref{thm:main}]
Let $L$ be a non-split, nontrivial link other than the Hopf link,
and let $D$ be a reduced link diagram of $L$ with $n$ crossings.
Under these assumptions, we have $n\ge 3$.
Let $\Gamma$ denote the underlying $4$-valent plane graph of $D$,
and let $T$ be a spanning tree of $\Gamma$.

Since $n\ge 3$, the tree $T$ contains a path of length two.
Let $e_1$ and $e_2$ denote the two $1$-cells (edges of $D$)
forming this path, sharing a common vertex $c$.
Each $1$-cell of $D$ is incident to exactly two $2$-cells of the induced
cell complex $X(D)$.
Thus, $e_1$ is incident to two $2$-cells, say $f_1$ and $f_2$,
and $e_2$ is incident to two $2$-cells, say $g_1$ and $g_2$.

We claim that among the four $2$-cells $f_1,f_2,g_1,g_2$,
at most one $2$-cell can be common to both sets.
Suppose, to the contrary, that $e_1$ and $e_2$ are incident to the same
two $2$-cells $f_1$ and $f_2$.

Since the graph $\Gamma$ is $4$-valent, there are exactly two further edges
of $\Gamma$ incident to the vertex $c$.
Denote these edges by $e_3$ and $e_4$.
Because $D$ is reduced, the edges $e_1,e_2,e_3,e_4$ are pairwise distinct.

Consider the checkerboard coloring of the complementary regions of $D$.
Along the boundary of any face, the incident edges must alternate
between black and white regions.
Since both $e_1$ and $e_2$ separate the same two faces $f_1$ and $f_2$,
the remaining edges $e_3$ and $e_4$ must also lie entirely on the boundary
of one of these two faces.
Without loss of generality, assume that both $e_3$ and $e_4$ are contained
in the boundary of $f_1$.

Then the face $f_1$ contains all four edges incident to the vertex $c$.
Consequently, removing the vertex $c$ disconnects the graph $\Gamma$,
so $c$ is a cut-vertex.
This contradicts the assumption that $D$ is reduced.

Therefore, there exist two distinct $2$-cells,
one incident to $e_1$ and the other incident to $e_2$,
that do not share a common $1$-cell.
In particular, these two $2$-cells are non-adjacent.

Now start with the spanning tree $T\subset X^{(1)}(D)$
and enlarge it by including these two non-adjacent $2$-cells.
Since the two $2$-cells do not share a common $1$-cell,
the resulting subcomplex is connected and contractible,
and hence forms an extended spanning tree in the sense of
Section~\ref{sec:cell_cpx}.

By Lemma~\ref{lem:extended}, each $2$-cell contained in an extended
spanning tree reduces the number of intersections between the binding
circle and the diagram by one.
In particular, the existence of at least one such $2$-cell implies that
the induced three-page presentation has at most $3n$ arcs,
and the existence of two non-adjacent $2$-cells yields
$$
\alpha_3(L)\le 3n-1 = 3c(L)-1.
$$
This completes the proof.
\end{proof}

In the next section, we do not pursue further improvements of the bound.
Instead, we reinterpret the above construction in graph-theoretic terms,
with the aim of clarifying the underlying combinatorial structure
and formulating natural conjectures suggested by this viewpoint.

\section{Graph-theoretic interpretation and further directions}\label{sec:graph}

In this section we reinterpret the construction of extended spanning trees
from Section~\ref{sec:cell_cpx} in graph-theoretic terms and outline
possible directions for improving the linear bound
$\alpha_3(L)\le 3c(L)$.
Our aim is not to establish new bounds here,
but to clarify the combinatorial structure underlying our method
and to formulate natural conjectures suggested by it.

Let $D$ be a reduced non-split link diagram with $n$ crossings,
and let $\Gamma(D)$ denote the planar dual graph of the underlying
$4$-valent plane graph of $D$.
Then $\Gamma(D)$ is a connected bipartite planar graph with $n+2$ vertices,
each corresponding to a complementary region of the diagram.

As observed in Section~\ref{sec:cell_cpx}, an improvement of the bound
on the three-page index is obtained by incorporating $2$-cells
into an extended spanning tree $Y\subset X(D)$.
From the viewpoint of the dual graph, this procedure admits a natural
graph-theoretic interpretation.
Indeed, $2$-cells of $X(D)$ correspond bijectively to vertices of
$\Gamma(D)$, and the condition that no two $2$-cells in $Y$
share a common $1$-cell translates into the requirement that the
corresponding vertices form an independent set in $\Gamma(D)$.

Furthermore, the requirement that $Y$ be connected and contractible
imposes a global constraint on the corresponding independent sets.
In terms of the dual graph $\Gamma(D)$, this means that removing the chosen
vertices must not disconnect the graph.
Thus, improving the upper bound for the three-page index reduces to the problem of finding large
\emph{non-separating independent sets}
in the planar bipartite graph $\Gamma(D)$.
Connectivity constraints under vertex deletion are classical topics in graph theory;
see, for instance, Diestel~\cite{Diest} for general background.

This reformulation suggests several graph-theoretic strategies
for constructing admissible independent sets in $\Gamma(D)$.
Among the most natural approaches are methods based on
spanning trees with many leaves in planar bipartite graphs.
Spanning trees with many leaves have been studied from both extremal
and algorithmic perspectives; see, for example,
Kleitman and West~\cite{KW}.
If a spanning tree of $\Gamma(D)$ admits a large collection of leaves
contained in a single bipartition class,
any subset of these leaves automatically forms an independent set.
If, moreover, the removal of such leaves preserves connectivity,
the corresponding vertices yield admissible $2$-cells
for an extended spanning tree.

While these strategies are conceptually straightforward,
their rigorous implementation involves subtle global constraints
arising from planarity and connectivity.
A systematic graph-theoretic treatment of these issues
appears to require further technical development
and is left for future work.

The considerations above naturally lead to the following conjectures.

\begin{conjecture}\label{conj:linear_improvement}
There exists a constant $c<3$ such that for every non-split link $L$,
$$
\alpha_3(L)\le c\,c(L)+O(1).
$$
\end{conjecture}

Equivalently, the problem may be formulated purely in terms of planar graphs.

\begin{conjecture}\label{conj:nonsep_independent}
There exists a constant $\delta>0$ such that for any reduced link diagram
$D$ with $n$ crossings, the dual graph $\Gamma(D)$ contains
a non-separating independent set of size at least $\delta n$.
\end{conjecture}

A positive resolution of Conjecture~\ref{conj:nonsep_independent}
would immediately yield an improved linear upper bound for the
three-page index via the extended spanning tree construction
of Section~\ref{sec:cell_cpx}.
We expect that this graph-theoretic perspective provides a useful viewpoint
for further investigations at the interface of
link theory and planar graph theory.

\section*{Acknowledgements}
This study was supported by Basic Science Research Program of the National Research Foundation of Korea (NRF) grant funded by the Korea government Ministry of Education (RS-2023-00244488).

\bibliography{threepage} 
\bibliographystyle{abbrv}
	
\end{document}